   \documentclass[twoside,reqno,11pt]{fcaa-var} %

\usepackage{graphicx}
\usepackage{epsfig}
\usepackage{amsthm}
\usepackage{amsmath}
\usepackage{latexsym}
\usepackage{amsfonts}
\usepackage{amssymb}

\newtheoremstyle{theorem}
  {15pt}          
  {15pt}  
  {\sl}  
  {\parindent}
  {\sc}  
  {. }   
  { }    
  {}     
\theoremstyle{theorem}
\newtheorem{lemma}{Lemma}[section]
\newtheorem{theorem}{Theorem}[section]

\newtheoremstyle{defi}
  {15pt}          
  {15pt}  
  {\rm}  
  {\parindent}     
  {\sc}  
  {. }    
  { }    
  {}     
\theoremstyle{defi}
\newtheorem{definition}{Definition}[section]
\newtheorem{remark}{Remark}[section]
\newtheorem{example}{Example}[section]

 \def\proofname{\indent\rm P\ r\ o\ o\ f.}
 \def\proofend{\hfill$\Box$}
 \def\qed{\hfill$\Box$} 

\usepackage{epstopdf}
\usepackage{graphicx}
\usepackage{subfigure}
\usepackage{multirow}
 \usepackage{hyperref} 
\usepackage{extarrows}
 \def\theequation{\arabic{section}.\arabic{equation}}

 \def\themycopyrightfootnote{\vspace*{3pt}
   \hfill  \vspace*{-36pt}
  }

 \title[Fractional asymptotical regularization]{On fractional asymptotical regularization of linear ill-posed problems in Hilbert spaces}
 \author[\normalsize Y. Zhang, B. Hofmann]{\normalsize Ye Zhang $^1$, Bernd Hofmann $^2$}

\newtheorem{proposition}{Proposition}

\begin{document}

 \vbox to 2.5cm { \vfill }


 \bigskip \medskip

 \begin{abstract}

In this paper, we study a fractional-order variant of the asymptotical regularization method, called {\it Fractional Asymptotical Regularization (FAR)}, for solving linear ill-posed operator equations in a Hilbert space setting. We assign the method to the general linear regularization schema and prove that under certain smoothness assumptions, FAR with fractional order in the range $(1,2)$ yields an acceleration with respect to comparable order optimal regularization methods. Based on the one-step Adams-Moulton method, a novel iterative regularization scheme is developed for the numerical realization of FAR. Two numerical examples are given to show the accuracy and the acceleration effect of FAR.

 \medskip


 \smallskip


 \end{abstract}

 \maketitle

 \vspace*{-16pt}


\section{Introduction}\label{Sec:1}

\setcounter{section}{1}
\setcounter{equation}{0}\setcounter{theorem}{0}

It belongs to the main areas of competence of the journal ``Fractional Calculus \& Applied Analysis'' to publish papers that handle initial-boundary value problems and abstract initial-value problems in Hilbert spaces for time-fractional diffusion and wave equations, partially also with respect to inverse and ill-posed problems. In this context, we refer for example to the papers \cite{KianYamamoto2017,LiLuchkoYamamoto2014,LiuRundellYamamoto2016,LuchkoYamamoto2016} by {\sc M.~Yamamoto, Y.~Luchko}, and coauthors. The behavior of solutions to such equations plays an important role for finding stable approximations to inverse problems aimed at characterizing parameter functions of the underlying physical processes. As the papers mentioned above show, error estimates and explicit representations of solutions to the fractional equations can also be used to generate new continuous regularization methods. In the present paper, we try to complement this research direction by extending the well-known method of asymptotical
regularization (Showalter's method) in Hilbert spaces to a fractional version using the left-side Caputo fractional derivative. In particular, we are going to derive regularization properties by a stringent analysis and to show
at least by means of case studies an acceleration effect based on the fractional derivatives.

Let $A$ be a compact linear operator acting between two infinite dimensional Hilbert spaces $\mathcal{X}$ and $\mathcal{Y}$ such that the range $\mathcal{R}(A)$ of $A$ is an infinite dimensional subspace of  $\mathcal{Y}$. It is
well-known that then $\mathcal{R}(A)$ is non-closed, i.e.~$\overline{\mathcal{R}(A)}^{\mathcal{Y}}\not=\mathcal{R}(A)$, and the linear operator equation
\begin{equation}\label{OperatorEq}
Ax=y,
\end{equation}
is {\it ill-posed} and requires some kind of {\it regularization}. This is typical for operator equations (\ref{OperatorEq}) which are models for {\it linear inverse problems} (cf., e.g.,~\cite{Groetsch2007} or \cite[\S~2]{Hansen2010} and references therein).
For simplicity, we denote by $\langle \cdot, \cdot \rangle$ and $\|\cdot\|$ in the sequel the inner products and norms for both Hilbert spaces $\mathcal{X}$ and $\mathcal{Y}$.
In this paper, we are interested in stable approaches for solving (\ref{OperatorEq}) that can be realized by fast iterative algorithms. The corresponding stable approximate solutions should be based on noisy data $y^\delta$ of the exact right-hand side $y=Ax^\dagger$, where $x^\dagger=A^\dagger y$ denotes the minimum-norm solution of (\ref{OperatorEq}). In this context, we consider the noise model $\|y^\delta - y\|\leq \delta$ with noise level $\delta>0$.

The simplest iterative regularization approach for solving \eqref{OperatorEq} seems to be the Landweber method, which is given by the iteration procedure
\begin{equation}\label{Landweber}
x^\delta_{k+1} = x^\delta_{k} + \Delta t A^* ( y^\delta- A x^\delta_{k} ), \quad \Delta t\in(0, 2/\|A\|^2) \quad (k=0,1,2...)
\end{equation}
with some starting element $x_0 \in\mathcal{X}$, where $A^*$ denotes the adjoint operator of $A$. The continuous analog to (\ref{Landweber}) as $\Delta t$ tends to zero is known as {\it asymptotic regularization} or {\it  } (see, e.g.,~\cite{Tautenhahn-1994,Vainikko1986}). It is written as a first order evolution equation of the form
\begin{eqnarray}\label{FisrtFlow}
\dot{x}^\delta(t) + A^* A x^\delta(t)= A^* y^\delta
\end{eqnarray}
with some initial condition, where an artificial scalar time $t$ is introduced. There must be chosen an appropriate finite {\it stopping time} $T_*=T_*(\delta)$ (a priori choice) or $T_*=T_*(\delta,y^\delta)$ (a posteriori choice)
in order to ensure the regularizing property $x^\delta(T_*) \to x^\dagger$ as $\delta\to0$.

For the asymptotic regularization it is well-known that for all $p>0$ under H\"{o}lder-type source conditions
\begin{equation}\label{HolderSourceIntro}
x^\dagger = (A^*A)^p v, \quad \|v\| \le \rho,
\end{equation}
and for the stopping time $T_*$ selected according to the a priori choice $T_*=T_*(\delta) \sim \delta^{-\frac{2}{2p+1}}$  we have for all $p>0$ (cf., e.g.,~\cite[Theorem 2]{Tautenhahn-1994}) the order optimal convergence rate
\begin{equation}\label{Convergence_Showalter}
\|x^\delta (T_*) - x^\dagger \| = \mathcal{O}(\delta^{\frac{2p}{2p+1}}) \quad  \mbox{as} \quad \delta \to 0.
\end{equation}
From the analog result of the Landweber iteration in \cite[Theorem 6.5]{engl1996regularization} it can be concluded for asymptotical regularization that we have for the stopping time $T_*=T_*(\delta,y^\delta)$ chosen according to Morozov's discrepancy principle the formulas
 \begin{equation}\label{Choice_Showalter}
 \|x^\delta (T_*) - x^\dagger \| = \mathcal{O}(\delta^{\frac{2p}{2p+1}}) \quad \mbox{and} \quad T_* = \mathcal{O}(\delta^{-\frac{2}{2p+1}}) \quad  \mbox{as} \quad \delta \to 0.
\end{equation}

Moreover, it has been shown that by using Runge-Kutta integrators, all of the properties of asymptotic regularization (\ref{FisrtFlow}) carry over to its numerical realization~\cite{Rieder-2005}. Hence, the continuous model (\ref{FisrtFlow}) is of particular importance for studying the intrinsic properties of a broad class of general linear regularization methods for inverse problems, and can be used for the development of new iterative regularization algorithms by combining some appropriate numerical schemes.

A fatal defect for large-scale problems is the slow performance of Landweber iteration (too many iterations required for optimal stopping)  as well as of the asymptotical regularization method, i.e.~too excessive stopping times $T_*$ are required for obtaining optimal convergence rates \eqref{Convergence_Showalter}. Therefore, in practice, accelerating strategies are usually used. The well-known methods are semi-iterative methods (e.g. the Brakhage's $\nu$-method) and the Nesterov acceleration scheme. It has been proven that
\begin{itemize}
\item for the $\nu$-method, the optimal convergence rates can be obtained with approximately the square root of iterations than those needed for ordinary Landweber iteration \cite[\S~6.3]{engl1996regularization}. However, in contrast to the Landweber iteration, the $\nu$-methods show a saturation phenomenon; i.e., the optimal convergence rate \eqref{Convergence_Showalter} and the asymptotic $k^* = \mathcal{O}(\delta^{-\frac{1}{2p+1}})$ hold only for $p\leq\nu$ and $p\leq\nu-0.5$, respectively.
\item For the Nesterov acceleration scheme, the optimal convergence rates are obtained if $p\leq1/2$ and if the iteration is terminated according to an a priori stopping rule. If $p>1/2$ or if the iteration is terminated according to the Morozov's conventional discrepancy principle, only sub-optimal convergence rates can be guaranteed. In both cases, sub-optimal convergence rates can be obtained with the same acceleration speed as by the $\nu$-method \cite{Neubauer-2017}.
\end{itemize}

Recently, inspired by these two accelerated regularization methods and the advantage of the continuous model, the authors extended \eqref{FisrtFlow} in \cite{ZhaHof2018AA} to the damped second order dynamics of the form
\begin{equation}\label{SecondFlow}
\ddot{x}^\delta(t) + \eta \dot{x}^\delta(t) + A^* A x^\delta(t)= A^* y^\delta, \quad \eta>0.
\end{equation}
It has been shown that, under condition \eqref{HolderSourceIntro}, \eqref{SecondFlow} exhibits the same convergence rate \eqref{Convergence_Showalter} as the asymptotic regularization method for all smoothness index $p>0$. By using the total energy discrepancy principle for choosing the terminating time and the damped symplectic integrators, a new iterative regularization algorithm has been proposed in this regard. Even though numerical experiments demonstrate the acceleration effect of the method \eqref{SecondFlow}, rigorous mathematical proof is still lacking.

In this paper, we replace the first derivative in the dynamical model \eqref{FisrtFlow} with appropriate fractional derivatives. Specifically, we consider in Hilbert spaces the following initial value problem
\begin{equation}\label{FractionalFlow}
\left( ^C D^\theta_{0+} x^\delta \right) (t) + A^*A x^\delta(t) = A^* y^\delta, \quad D^{k} x^\delta (0) =b_k,~k=0, ..., n-1,
\end{equation}
to an evolution equation of fractional order $\theta$,
where $\theta\in(0,2), n=\lfloor\theta\rfloor+1$, and $\lfloor\cdot\rfloor$ is the floor function. $D^{k}$ denotes the usual differential operator of order $k$. The left-side Caputo fractional derivative is defined by $\left( ^C D^\theta_{0+} x \right) (t) := I^{n-\theta}_{0+} D^n x(t)$, where $I^{n-\theta}_{0+}$ is the left-side Riemann-Liouville integral operator, i.e., $(I^{n-\theta}_{0+} x)(t):= \frac{1}{\Gamma(\theta)} \int^t_0 \frac{x(t)}{(t-\tau)^{1-n+\theta}} d\tau$ with the gamma function $\Gamma(\cdot)$. Note that, for $\theta=1$, \eqref{FractionalFlow} coincides with Showalter's method (asymptotical regularization). We will call the fractional dynamics \eqref{FractionalFlow} with an appropriate choice of terminating time {\it Fractional Asymptotical Regularization (FAR)}. The main goal of this paper is to show that under smoothness assumptions imposed on the exact solution, FAR with $\theta\in(1,2)$ yields an accelerated optimal regularization method.

The paper is structured as follows: in Section \ref{sec:2}, we review basic concepts in the general linear regularization theory so that they can be applied to the convergence analysis of FAR. The regularization properties of FAR under H\"{o}lder-type and logarithmic source conditions are studied in Sections \ref{sec:3} and \ref{logarithmic}, respectively. Section \ref{NumericalScheme} is devoted to the numerical realization of FAR. Finally, concluding remarks are given in Section \ref{Conclusion}.

\section{Linear regularization methods revisited}\label{sec:2}

\setcounter{section}{2}
\setcounter{equation}{0}\setcounter{theorem}{0}


We start with some definitions, taken from \cite{Mathe-2003} and \cite{Hofmann-2007}, respectively.

 \begin{definition}\label{Index}
A function $\varphi: \mathbb{R}_+ \to \mathbb{R}_+$ is called an index function if it is continuous, strictly increasing and satisfies $\lim_{\lambda\to 0+} \varphi(\lambda)=0$.
 \end{definition}

 \begin{definition}\label{gAlpha}
A family of functions $g_\alpha (\lambda)\;(0<\lambda\le \|A\|^2)$, defined for regularization parameters $0<\alpha \le \bar \alpha$, is called a generator function for a linear regularization method
to the ill-posed linear operator equation \eqref{OperatorEq} if the following three conditions are fulfilled:
\begin{itemize}
\item[(i)]  For the bias function $r_\alpha(\lambda) = 1 - \lambda g_\alpha(\lambda)$ we have for any fixed $\lambda\in (0,\|A\|^2]$ the limit condition $\lim_{\alpha\to0} |r_{\alpha}(\lambda)| =0$.
 \item[(ii)] There exists a constant $\gamma_1 >0 $ such that $|r_{\alpha}(\lambda)|\leq \gamma_1$ for all $\lambda\in(0,\|A\|^2]$ and for all $\alpha\in(0,\bar{\alpha}]$.
\item[(iii)] There exists a constant $\gamma_*>0$ such that $\sqrt{\lambda} |g_{\alpha}(\lambda)| \leq \gamma_*/\sqrt{\alpha}$ for all $\alpha\in(0,\bar{\alpha}]$.
\end{itemize}
 \end{definition}

Once a generator function family $g_\alpha (\lambda)$ is chosen, the approximate solutions to \eqref{OperatorEq} based on the noisy data $y^\delta$ are calculated by the procedure
\begin{equation}\label{regularization}
x^\delta_{\alpha} = g_\alpha(A^* A)A^* y^\delta,
\end{equation}
where
\begin{equation}\label{NoiseFreeSolu}
x_{\alpha} = g_\alpha(A^* A)A^* A x^\dagger,
\end{equation}
characterizes the noise-free analog to $x_\alpha^\delta$. By condition (iii) of Definition~\ref{gAlpha} and exploiting the triangle inequality, we obtain the well-known error estimates
\begin{equation}\label{errorBounds2}
\begin{array}{ll}
\|x^\delta_{\alpha}-x^\dagger\| & \leq  \|x_{\alpha}-x^\dagger\| + \|x^\delta_{\alpha}-x_{\alpha}\|
\\ &
= \|r_\alpha(A^* A) x^\dagger\| + \delta \|g_\alpha(A^* A)A^*\| \\ &
\leq \|r_\alpha(A^* A) x^\dagger\| + \gamma_* \delta / \sqrt{\alpha}.
\end{array}
\end{equation}
Furthermore, from properties (i) and (ii) of Definition \ref{gAlpha} we deduce for $\alpha\to0$ point-wise convergence $r_\alpha(A^* A)x\to0$ for any $x\in \mathcal{X}$ (see, e.g.,~\cite[Theorem 4.1]{engl1996regularization}). If the regularization parameters $\alpha=\alpha(\delta)$ or $\alpha=\alpha(\delta,y^\delta)$ are chosen such that
$$\lim_{\delta\to 0} \alpha =  \lim_{\delta\to 0} \delta^2 / \alpha = 0,$$
one can derive convergence $\|x^\delta_{\alpha}-x^\dagger\|\to 0$ as $\delta\to0$
by the estimate \eqref{errorBounds2}.

If we have an index function $f$ depending on $x^\dagger$ as upper bound of the form
\begin{equation} \label{eq:profile}
\|r_\alpha(A^* A) x^\dagger\|\leq f(\alpha) \quad (0<\alpha \le \bar \alpha)
\end{equation}
then this function also determines the error profile for the solution $x^\dagger$ in the noisy case, and was therefore termed `\emph{profile function}' in \cite{Hofmann-2007}. By
using the auxiliary index function $\Theta(\alpha) := \sqrt{\alpha} f(\alpha)$ and choosing the regularization parameter a priori as $\alpha_* = \Theta^{-1}(\delta)$, we can derive under additional
conditions the convergence rate
\begin{equation}\label{Estimator3}
\|x^\delta_{\alpha_*}-x^\dagger\| \leq (1 + \gamma_*) f(\Theta^{-1}(\delta))=\mathcal{O}\left(f(\Theta^{-1}(\delta))\right)  \quad \mbox{as} \quad \delta \to 0.
\end{equation}

Note that the three requirements in Definition \ref{gAlpha} are not sufficient to yield profile functions $f$ for classes of solutions $x^\dagger$. It is well-known that convergence rates for approximate solutions of form \eqref{regularization} are connected with smoothness conditions imposed on the solution $x^\dagger$ with respect to the forward operator $A$. In order to measure the sensitivity of a regularization method with respect to possible smoothness assumptions, the concept of qualification in the sense of index functions is introduced:
\begin{definition}
\label{QualificationDef}
 Let $\varphi$ be an index function. A regularization method \eqref{regularization} generated by $g_\alpha (\lambda)$ is said to have the qualification $\varphi$ if there are a constant $\gamma>0$, independent of $\alpha$,
 and a value $\bar\alpha>0$ such that for all $\alpha\in(0,\bar{\alpha}]$ the inequality
\begin{equation*}\label{QualificationFun}
\sup_{\lambda\in(0,\|A\|^2]} |r_{\alpha}(\lambda)| \varphi(\lambda) \leq \gamma \varphi(\alpha)
\end{equation*}
is satisfied.
\end{definition}
It should be noted that the concept of Definition~\ref{QualificationDef} is an alternative to the traditional concept of qualification in the sense of a positive number or infinity used in  \cite{Vainikko1986}
and \cite[Chap.~4]{engl1996regularization}. This alternative concept was originally introduced in \cite{Mathe-2003}, but it has been frequently used by other authors recently. Only if the index function $\varphi$ is a qualification
in the sense of Definition~\ref{QualificationDef} for the regularization method with generator functions $g_\alpha$, we can derive the rate result (\ref{Estimator3}) from (\ref{eq:profile}).

\begin{example}[H\"older source conditions]\label{ex:ex1}
In this example we consider the index functions $\varphi(\lambda)=\lambda^p$ associated with the H\"older source conditions (\ref{HolderSourceIntro}). If such $\varphi$ is for some $p>0$ a qualification for method generated by $g_\alpha$, then we obtain from (\ref{errorBounds2})
$$\|x^\delta_{\alpha}-x^\dagger\|\le \gamma \alpha^p \rho + \gamma_* \delta / \sqrt{\alpha} $$
and hence for $\alpha_*=\alpha(\delta)=c \delta^{2p/(2p+1)}\;(c>0)$ the error estimate and convergence rate
\begin{equation} \label{eq:estpower}
\|x_{\alpha_*}^\delta-x^\dagger\|\le \left(\gamma \,c^p+ \gamma_*/ \sqrt{c} \right) \delta^{\frac{2p}{2p+1}} = \mathcal{O}\left(\delta^\frac{2p}{2p+1}\right) \quad \mbox{as} \quad \delta \to 0.
\end{equation}
\end{example}

If a benchmark source condition $\psi$ is known to be a qualification for the method generated by $g_\alpha$, then other index functions $\varphi$ are also qualifications whenever they are {\it covered} by $\psi$,
and we refer to \cite[Def.~2]{Mathe-2003} and \cite[Prop.~3, Remark~5 and Lemma~2]{Mathe-2003} for the following definition and proposition, respectively.
\begin{definition}
\label{def:covered}
Let $\psi(\lambda)\;(0<\lambda\le \|A\|^2)$ be an index function. Then an index function $\varphi(\lambda)\;(0<\lambda\le \|A\|^2)$ is said to be covered by $\psi$ if there is $\underline c>0$ such that
 \begin{equation*}\label{QualificationCover}
\underline c\,\frac{\psi(\alpha)}{\varphi(\alpha)} \le \inf \limits _{\alpha \le \lambda \le \|A\|^2} \frac{\psi(\lambda)}{\varphi(\lambda)} \qquad (0 < \alpha \le \bar \alpha).
\end{equation*}
\end{definition}

\begin{proposition} \label{pro:covered}
The index function $\varphi$ is a qualification for the method generated by $g_\alpha$ if $\varphi$ is covered by $\psi$ and if $\psi$ is a qualification for that method. If the quotient function $\lambda \mapsto  \frac{\psi(\lambda)}{\varphi(\lambda)}$ is increasing for $0<\lambda \le \bar \lambda$ and some $\bar \lambda>0$, then $\varphi$ is covered by $\psi$. If, in particular, the index function $\varphi(\lambda)$ is concave for
$0<\lambda \le \bar \lambda$, then $\varphi$ is covered by $\psi(\lambda)=\lambda$.
\end{proposition}

\begin{example}[Logarithmic source conditions]\label{ex:ex2}
The focus of this example is on logarithmic source conditions
\begin{equation}\label{logarithmicCondition}
x^\dagger = \varphi_\mu(A^*A) v, \quad \|v\|\leq \rho,
\end{equation}
where for exponents $\mu>0$ the function $\varphi_\mu$ is defined as
\begin{equation}\label{logarithmicQualification}
\varphi_{\mu}(\lambda) = \left\{\begin{array}{ll}
\log^{-\mu}(1/\lambda) \qquad \textrm{~for} \quad 0< \lambda  \le e^{-\mu-1},  \\
\mbox{arbitrarily extended as index function for} \;\;\lambda > e^{-\mu-1}.
\end{array}\right.
\end{equation}
Note that the  condition \eqref{logarithmicCondition} for any $\mu>0$ is significantly weaker than the condition \eqref{HolderSourceIntro}, even for arbitrarily small $p>0$. As a consequence, the expected logarithmic convergence rates are also significantly lower than in the H\"older case \eqref{eq:estpower} and typically occur for severely ill-posed problems (cf., e.g.,~\cite{Hohage2000}).
If $\varphi_\mu$ is for some $\mu>0$ a qualification for method generated by $g_\alpha$, then we obtain from (\ref{errorBounds2})
$$\|x^\delta_{\alpha}-x^\dagger\|\le \gamma\, \varphi_\mu(\alpha)\, \rho + \gamma_* \delta / \sqrt{\alpha}, $$
for sufficiently small $\alpha>0$, and hence for $\alpha_*=\alpha(\delta)=c \,\delta^\kappa\;(c>0, \linebreak 0<\kappa<2)$ the convergence rate
\begin{equation} \label{eq:estlogarithmic}
\|x_{\alpha_*}^\delta-x^\dagger\| = \mathcal{O}\left(\varphi_\mu(\delta)\right) \quad \mbox{as} \quad \delta \to 0
\end{equation}
as a consequence of $\lim_{\delta \to 0} \delta^\eta/\varphi_\mu(\delta)=0$, which is valid for all $\eta>0$.
\end{example}

\section{FAR under H\"{o}lder-type source conditions}
\label{sec:3}

\setcounter{section}{3}
\setcounter{equation}{0}\setcounter{theorem}{0}

In this section, we show that the FAR method with fractional order $\theta>0$  introduced in the introduction can be assigned to the general linear regularization schema recalled in Section~\ref{sec:2}.
Therefore, we start with the verification of the generator function $g_\alpha$ occurring in formula \eqref{regularization} associated with the fractional differential equation \eqref{FractionalFlow} and the corresponding regularization properties. For simplicity, let $b_k=0$ in \eqref{FractionalFlow} for all $k$ under consideration. The case with non-vanishing initial data can be analyzed similarly, see \cite{ZhaHof2018AA}.

Let $\{\sigma_j; u_j, v_j\}_{j=1}^\infty$ be the well-defined singular system for the compact linear operator $A$, i.e.~we have $Au_j= \sigma_j v_j$ and $A^* v_j = \sigma_j u_j$ with ordered singular values $\|A\|=\sigma_1 \geq \sigma_2 \geq \cdot\cdot\cdot \geq \sigma_j \geq \sigma_{j+1} \geq \cdot\cdot\cdot \to 0$ as $j \to \infty$. Since the eigenelements $\{u_j\}_{j=1}^\infty$ and $\{v_j\}_{j=1}^\infty$ form complete orthonormal systems (with the exception of null-spaces) in  $\mathcal{X}$ and $\mathcal{Y}$,
respectively, \eqref{SecondFlow} is equivalent to
\begin{equation*}\label{SVDEq1}
\left\langle \left( ^C D^\theta_{0+} x^\delta \right) (t) , u_j  \right\rangle + \sigma^2_j \langle x^\delta(t) , u_j \rangle = \sigma_j \langle y^\delta , v_j \rangle, \quad j=1,2, ...\,.
\end{equation*}
Using the decomposition $x^\delta(t)=\sum \limits_{j=1}^\infty \xi_j(t) u_j$, we obtain
\begin{equation}\label{SVDEq2}
\left( ^C D^\theta_{0+} \xi_j \right) (t) + \sigma^2_j \xi_j(t) = \sigma_j \langle y^\delta , v_j \rangle, \quad j=1,2, ...\,.
\end{equation}

On the other hand, according to \cite[\S~4.1.3]{Kilbas2006}, the solution to the equation $\left( ^C D^\theta_{0+} \eta \right) (t) - \lambda \eta(t) = \kappa(t)$ with the vanishing initial data is explicitly given by (see formula (4.1.62) in \cite{Kilbas2006})
\begin{equation}\label{SVDSoluOld}
\eta(t) = \int^t_0 (t-\tau)^{\theta-1} E_{\theta,\theta}(\lambda(t-\tau)^\theta) \kappa(\tau) d \tau,
\end{equation}
where the two-parametric Mittag-Leffler function $E_{\theta_1,\theta_2}(z)$ is defined as $$E_{\theta_1,\theta_2}(z) = \sum^\infty_{k=0} \frac{z^k}{\Gamma(\theta_1 k + \theta_2)}.$$

Therefore, together with \eqref{SVDSoluOld}, we deduce that the solution of \eqref{SVDEq2} with the vanishing initial data is given by
\begin{equation*}\label{SolutionXi1}
\xi_j(t) = \sigma_j \langle y^\delta , v_j \rangle \int^t_0 (t-\tau)^{\theta-1} E_{\theta,\theta}(-\sigma^2_j (t-\tau)^\theta) d \tau,
\end{equation*}

Consequently, by the above result, together with the identity \cite[\S~4.4]{Gorenflo2014}
\begin{equation*}
\int^t_0 E_{\theta_1,\theta_2}(\xi z^{\theta_1}) z^{\theta_2 -1} dz = t^{\theta_2} E_{\theta_1,\theta_2+1}(\xi t^{\theta_1}), \quad \theta_2>0,
\end{equation*}
we deduce for all $\theta>0$ that $$\xi_j(t) = \sigma_j t^\theta E_{\theta,\theta+1}(-\sigma^2_j t^\theta) \langle y^\delta , v_j \rangle, \quad j=1,2, ...\,.$$
 Thus, by using $x^\delta(t)=\sum_{j=1}^\infty \xi_j(t) u_j$, we obtain the explicit formula
\begin{equation}\label{SolutionX}
x^\delta(t) = t^\theta \sum^\infty_{j=1}  E_{\theta,\theta+1}(-\sigma^2_j t^\theta) \sigma_j \langle y^\delta , v_j \rangle u_j = g^\theta(t, A^* A) A^* y^\delta
\end{equation}
for the solution to \eqref{SecondFlow}, where the generator function $g^\theta(t,\lambda)$ characterizing the FAR method attains the form
\begin{equation}\label{g}
g^\theta(t, \lambda) := t^\theta E_{\theta,\theta+1}(-\lambda t^\theta).
\end{equation}

Furthermore, by using the recurrence relations
\begin{equation*}
z^m E_{\theta_1,\theta_2+m\theta_1} (z) = E_{\theta_1,\theta_2} (z) - \sum^{m-1}_{k=0} \frac{z^k}{\Gamma(\theta_1 k + \theta_2)}
\end{equation*}
we obtain an associated bias function $r^\theta(t,\lambda)=1-\lambda g^\theta(t,\lambda)$ of the form
\begin{equation}\label{r}
r^\theta(t, \lambda) = 1- \lambda t^\theta E_{\theta,\theta+1}(-\lambda t^\theta) = E_{\theta,1}(-\lambda t^\theta)\equiv E_{\theta}(-\lambda t^\theta),
\end{equation}
where $E_{\theta}(z)= \sum^\infty_{k=0} \frac{z^k}{\Gamma(\theta k + 1)}$ denotes the classical Mittag-Leffler Function.

\begin{remark}\label{RekTheta}
For $\theta=0.5$, we have $r(t, \lambda) = e^{-\lambda^2 t} \left[ 1 + \textrm{erf}(-\lambda \sqrt{t}) \right]$, where $\textrm{erf}(z):= \frac{2}{\sqrt{\pi}} \int^z_0 e^{-\tau^2} d\tau$ represents the error function. Moreover, $r(t, \lambda) = e^{-\lambda t}$ for $\theta=1$, while $r(t, \lambda) = \cos(\sqrt{\lambda}t)$ when $\theta=2$.
\end{remark}


\begin{lemma}\label{MittagLefflerIneqAll}
Let $\theta\in(0,2)$ be a fixed number. Then we have for all $z\in \mathbb{R}_+$ the following three inequalities
\begin{eqnarray}
|E_{\theta,\theta+1}(-z)|  &\leq& C_\theta/(1+ z), \label{EIneqG} \\
|E_{\theta}(-z)|  &\leq& C_\theta/(1+ z),  \label{EIneqR} \\
|E_{\theta}(-z)|  &\leq& 3, \label{EIneqRBounded}
\end{eqnarray}
where $C_\theta>0$ is a constant, depending only on $\theta$.
\end{lemma}

 \proof 
Inequalities \eqref{EIneqG} and \eqref{EIneqR} follow from \cite[Theorem 1.6]{Podlubny1999} and \cite[Corollary 3.7]{Gorenflo2014} respectively. For the estimate \eqref{EIneqRBounded}, we distinguish three cases: (I) $\theta\in(0,1)$, (II) $\theta=1$, and (III) $\theta\in(1,2)$. In this first case, \eqref{EIneqRBounded} holds according to the uniform estimate \cite[Theorem 4]{Simon2014}:
\begin{equation}\label{EIneq}
\frac{1}{1+ \Gamma(1-\theta) z} \leq E_{\theta}(-z)  \leq \frac{1}{1+ \Gamma(1+\theta)^{-1} z},
\end{equation}
For the second case, we have $E_{1}(-z) = e^{- z}\leq1$. Now, we consider the last case. To this end, we set $z:=t^\theta$ and consider the function $E_{\theta}(-t^\theta)$. It is well known (see, e.g.,~\cite{Mainardi2000})
that $E_{\theta}(-t^\theta) = f_{\theta}(t) + h_{\theta}(t)$, where the completely monotone function $-f_{\theta}$ is defined by
\begin{equation}\label{fTheta}
f_{\theta}(t)= \frac{1}{\pi} \int^\infty_0 e^{-rt} \frac{r^{\theta-1} \sin(\theta\pi)}{r^{2\theta}+2r^{\theta}\cos(\theta\pi)+1} dr,
\end{equation}
and the oscillatory part $h_{\theta}$ is given by
\begin{equation*}
h_{\theta}(t)= \frac{2}{\theta} e^{t\cos(\pi/\theta)} \cos\left( t \sin(\pi/\theta) \right).
\end{equation*}
Obviously, for $\theta\in(1,2)$, $f_{\theta}$ is a monotonically increasing function such that $-1< 1- \frac{2}{\theta} = \frac{1}{\pi} \int^\infty_0 \frac{r^{\theta-1} \sin(\theta\pi)}{r^{2\theta}+2r^{\theta}\cos(\theta\pi)+1} dr \leq f_{\theta}(t) \leq 0$, while $|h_{\theta}(t)|\leq 2/\theta$ as $\cos(\pi/\theta)<0$. This implies $|E_{\theta}(-t^\theta)|\leq 3$.
 \proofend 

We remark that $C_\theta\to1$ as $\theta\searrow1$, and $C_\theta\to\infty$ as $\theta\nearrow2$.

In order to assign FAR to the general linear regularization schema introduced in Section~\ref{sec:2}, we exploit the one-to-one correspondence between the artificial time $t>0$ and the conventional regularization parameter
$\alpha>0$ by setting $t:=\alpha^{-1/\theta}$.

\begin{theorem}\label{FARRegu}
For $\theta\in(0,2)$, $g_\alpha (\lambda):= g^\theta(\alpha^{-1/\theta}, \lambda)$ represents a generator function in the sense of Definition~\ref{gAlpha} with $\theta_*=C_\theta/2$ and $C_\theta$ from
Lemma~\ref{MittagLefflerIneqAll}. Consequently, FAR is a linear regularization method for the ill-posed linear operator equation \eqref{OperatorEq}.
\end{theorem}

 \proof 
We have to check the three requirements (i), (ii) and (iii) in Definition~\ref{gAlpha} for all $\theta\in(0,2)$.

Since $r_\alpha(\lambda) = r^{\theta}(\alpha^{-1/\theta}, \lambda) = E_{\theta}(-\lambda \alpha^{-1})$, it follows from \eqref{EIneqR} that
\begin{equation}\label{rEstimates}
|r_\alpha(\lambda)|  \leq  C_\theta \alpha/(\alpha + \lambda),
\end{equation}
which shows that (i) is satisfied.
The second condition (ii) is a consequence of \eqref{EIneqRBounded}. It remains to find a bound $\gamma_*$ in (iii). Note that $g_{\alpha}(\lambda)=\alpha^{-1} E_{\theta,\theta+1}(-\lambda \alpha^{-1})$. By using \eqref{EIneqG}, we obtain the inequalities
\begin{equation}\label{IneqCondition3}
\sqrt{\lambda} |g_{\alpha}(\lambda)| \leq \sqrt{\lambda} C_\theta /(\alpha + \lambda)  \leq C_\theta / (2\sqrt{\alpha})
\end{equation}
and hence $\gamma_*=C_\theta /2$, which completes the proof.
 \proofend 

Referring back to the time variable $t$  we obtain a linear regularization method for the ill-posed operator equation \eqref{OperatorEq} with the procedure
\begin{equation}\label{regularization2}
x^\delta(t) = g_{t^{-\theta}}(A^* A)A^* y^\delta.
\end{equation}

\begin{figure}[!t]
\caption{Behavior of bias function $r_\alpha(\lambda)$ for varying $\theta$.}
\subfigure{
\includegraphics[width=0.97\textwidth]{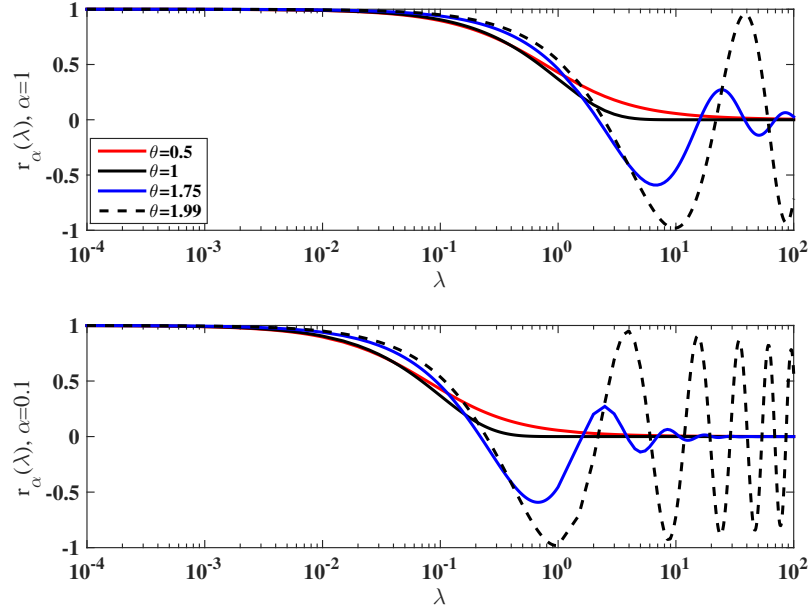}
}
\label{FigbiasFun}
\end{figure}

The behavior of the
bias function $r_\alpha(\lambda)$, or equivalently $r_{t^{-\theta}}(\lambda)$, is visualized in Figure \ref{FigbiasFun}. It follows from Remark \ref{RekTheta} that for $\theta=2$, $r_{t^{-2}}(\lambda) = \cos(\sqrt{\lambda t^{2}})$. Obviously, we have for all $\lambda$ under consideration $r_{t^{-2}}(\lambda) \not\to0$ as $t\to\infty$, which damages the first condition of Definition \ref{gAlpha}.  In order to obtain a regularization method in the case $\theta=2$, an additional damping term in the model \eqref{FractionalFlow} should be introduced, see \cite{ZhaHof2018AA} for details.


\begin{proposition}\label{ThmQualificationPower}  For all $\theta\in(0,1)\cup(1,2)$ the function $\varphi(\lambda)=\lambda^p$ is a qualification of the FAR method if and only if $0<p \le 1$, i.e.~there is a saturation for $p=1$.
\end{proposition}

 \proof 
The qualification of the FAR can be deduced from the estimate \eqref{EIneqR} and the following inequalities
\begin{equation}\label{IneqPower}
r_\alpha(\lambda) = \left| E_{\theta}(-\lambda \alpha^{-1}) \lambda^p  \right| \leq \frac{C_\theta \lambda^p}{1+ \lambda \alpha^{-1}} \leq \gamma(p,\theta) \alpha^{\min(p,1)},
\end{equation}
where $\gamma(p,\theta)=  C_\theta p^p(1-p)^{1-p}$ for $0<p<1$, and $\gamma(p,\theta)=  C_\theta \|A\|^{2(p-1)}$ for $1\leq p<2$. The saturation at $p=1$ follows from the lower bound in \eqref{EIneq} and the asymptotical behavior, see e.g. \cite[\S~1.2]{Podlubny1999}, $E_\theta(-z)z\to 1/\Gamma(1-\theta)$ as $z\to\infty$.
 \proofend 

\begin{remark}
By the well-known results of Showalter's method, see e.g. \cite[Remark 2.2]{ZhaHof2018AA}, $\varphi(\lambda)=\lambda^p$ is a qualification of the FAR method with parameter $\theta=1$ for all $p>0$ without saturation.
\end{remark}

\begin{theorem}\label{FARTikhonov}
Let $x(t)$ and $x^\delta(t)$ be solutions of \eqref{FractionalFlow} with noise-free and noisy data, respectively. Then, under the assumption $x^\dagger \in \mathcal{R}((A^*A)^p)$ (cf.~(\ref{HolderSourceIntro})), we have for the FAR method in the case $\theta\in(0,1)\cup(1,2)$ and $p\in(0,1]$ the convergence rate
\begin{equation}\label{ErrorEstimatePriori1Sec2}
\| x^\delta(T_*) - x^\dagger \| = \mathcal{O}\left(  \delta^{\frac{2p}{2p+1}} \right) \quad \mbox{as} \quad \delta \to 0
\end{equation}
whenever the terminating time $T_*=T_*(\delta)$  is chosen according to
\begin{equation}\label{ErrorEstimatePriori1Sec2T}
T_* (\delta) \sim  \delta^{-\frac{2}{\theta (2p+1)}}.
\end{equation}
Moreover, the following additional assertions are true if $\theta\in(0,1)$:
\begin{itemize}
\item[(i)] The convergence rate of method error $\|x(t)-x^\dagger\|=\mathcal{O}(t^{-\theta p})$ holds, if and only if $x^\dagger$ satisfies a source condition of type \eqref{HolderSourceIntro}.
\item[(ii)] The worst-case convergence rate
\begin{equation}\label{WorstErr}
\sup_{y^\delta: \|y^\delta - y\|\leq \delta} \inf_{t\in \mathbb{R}_+} \| x^\delta(t;y^\delta) - x^\dagger \|^2 = \mathcal{O}\left(  \delta^{\frac{2p}{2p+1}} \right)
\end{equation}
is equivalent to the following weaker condition
\begin{equation}\label{WeakerSource}
\|E_s x^\dagger\|^2:= \int^s_0 1 d \|E_\lambda x^\dagger\|^2 = \mathcal{O}\left(  s^{2p} \right).
\end{equation}
\item[(iii)] If a terminating time $T_*$ exists such that
\begin{equation}\label{WorstErr2}
\sup_{y^\delta: \|y^\delta - y\|\leq \delta} \| x^\delta(T_*;y^\delta) - x^\dagger \|^2 = o\left(  \delta^{\frac{2p}{2p+1}} \right),
\end{equation}
then $x^\dagger=0$.
\end{itemize}
\end{theorem}

 \proof 
The assertions concerning the convergence rate \eqref{ErrorEstimatePriori1Sec2} follow directly from (\ref{eq:estpower}) by setting $T_*(\delta)=(\alpha_*)^{-1/\theta}$ in combination with Proposition~\ref{ThmQualificationPower}. The remaining assertions (i), (ii) and (iii) of the theorem are based on standard arguments of regularization theory (\cite[\S~4.2]{engl1996regularization} and \cite{FHM11}) in light of estimates \eqref{EIneqG}, \eqref{EIneq}, and \eqref{rEstimates}.
 \proofend 

\begin{remark}
(a) The first part of Theorem \ref{FARTikhonov} means that the optimal convergence rates \eqref{ErrorEstimatePriori1Sec2} of FAR can be obtained with approximately $1/\theta$ of iterations than needed for ordinary asymptotical regularization method \cite{engl1996regularization}. Therefore, FAR with $\theta\in(1,2)$ yields an accelerated regularization method for all $p>0$. As the estimate \eqref{eq:estpower}
shows, for fixed $\delta>0$ the error of regularization is proportional to $\gamma$ and $\gamma_*$. These constants, however, tend to infinity as $\theta \nearrow 2$ such that values $\theta$ close to $2$ need
not be the best choices of this parameter.
(b) The source condition \eqref{HolderSourceIntro} implies the weaker source condition \eqref{WeakerSource}. The converse is in general not true, however, \eqref{WeakerSource} implies $x^\dagger \in \cap_{\nu<p}  \mathcal{R}((A^*A)^\nu)$, see \cite[Lemma 4.12]{engl1996regularization}.
\end{remark}

In practice, the a priori stopping rule \eqref{ErrorEstimatePriori1Sec2T} in Theorem \ref{FARTikhonov} is not realistic, since a good terminating time $T_*$ requires in addition to the knowledge of $\delta$
also the knowledge of $p$ characterizing the specific smoothness of the unknown exact solution $x^\dagger$. Therefore, a posteriori parameter choices of the stopping parameter are preferred, and we
consider Morozov's discrepancy principle as the most prominent version exploiting zeros of the discrepancy function
\begin{equation}\label{fractionalDP}
\chi(t) = \|Ax^\delta(t)-y^\delta\| - \tau \delta,
\end{equation}
where we assume $\tau> 3$ for the occurring factor of the noise level $\delta$.

\begin{lemma}\label{Rootdiscrepancy}
If for $\tau>3$ the condition $\|y^\delta\|> \tau \delta$ is satisfied, the function $\chi(T)$ has at least one root.
\end{lemma}

 \proof 
By the explicit formula of $x^\delta(t)$, c.f. \eqref{regularization2}, we derive together with the estimates \eqref{EIneqRBounded} and \eqref{IneqPower} that
\begin{equation*}
\begin{array}{ll}
\chi(t) &= \|Ax^\delta(t)-y^\delta\| - \tau  \delta = \|r_{t^{-\theta}}(AA^*)(y-(y-y^\delta)) \| - \tau  \delta  \\ &
\leq \|r_{t^{-\theta}}(AA^*) A x^\dagger \| +  \|r_{t^{-\theta}}(AA^*)(y-y^\delta)\| - \tau  \delta  \\ &
\leq \|x^\dagger \| \sup_{\lambda\in(0,\|A\|^2]} |r_{t^{-\theta}}(\lambda)|\sqrt{\lambda}  - (\tau -3) \delta \\ &
\leq \gamma(p,\theta) \|x^\dagger \| t^{-\theta/2} - (\tau -3) \delta  \to -(\tau -3)  \delta <0
\end{array}
\end{equation*}
as $t\to \infty$. The continuity of $\chi(t)$ is obvious as we are dealing with the linear problem.  Since $\chi(0)= \|y^\delta\|- \tau  \delta>0$, the existence of the root of $\chi(t)$ follows from the Bolzano's theorem.
 \proofend 

\begin{remark}
As the proof of Lemma \ref{Rootdiscrepancy} indicates, the function $\chi(t)$ is bounded above by a decreasing function $\zeta(t):=\gamma(p,\theta) \|x^\dagger \| t^{-\theta/2} - (\tau -3) \delta$. Thus, the trend of $\chi(t)$ is to be a decreasing function, where oscillations may occur, and we refer to Figure~\ref{Fig:Discrepancy} for illustration.
\end{remark}

\begin{figure}[!htb]
\caption{The behavior of $\chi(T)$ from (\ref{fractionalDP}) with different fractional orders $\theta$.}
\includegraphics[width=1.\textwidth]{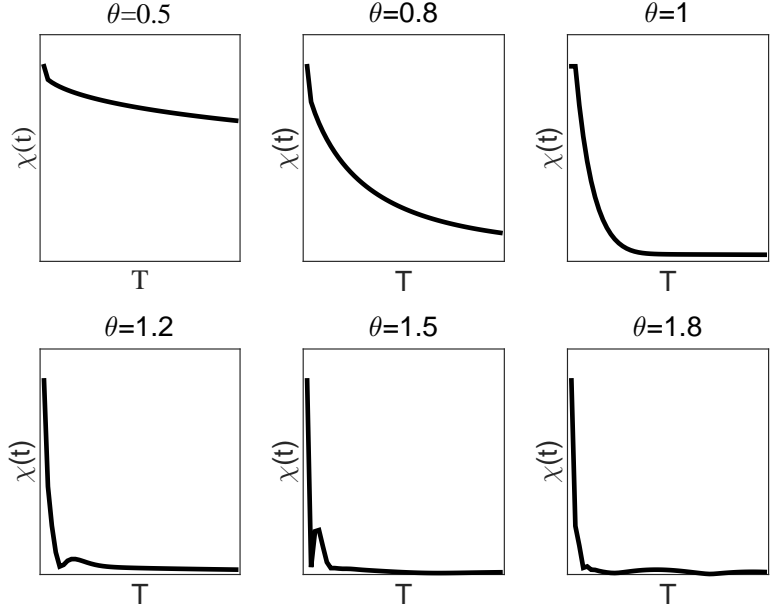}
\label{Fig:Discrepancy}
\vspace{-1cm}
\end{figure}

\begin{theorem}\label{ThmPosteriori}
Under the assumption $x^\dagger \in \mathcal{R}((A^*A)^p)$,  we have for the terminating time $T_*=T_*(\delta,y^\delta)$ of FAR chosen as a solution of $\chi(T)=0$ (cf.~\ref{fractionalDP}) the convergence rates
\begin{equation}\label{ErrorEstimatePrioriT}
T_* = \mathcal{O} \left( \delta^{-\frac{2}{\theta(2p+1)}} \right) \qquad \mbox{and}   \qquad \| x^\delta(T_*) - x^\dagger \| = \mathcal{O} \left( \delta^{\frac{2p}{2p+1}}  \right)
\end{equation}
for $p\in (0,1]$ if $\theta\in(0,1)\cup (1,2)$ and for all $p>0$ if $\theta=1$.
\end{theorem}

\begin{proof}
Using the interpolation inequality $\|B^p u\| \leq \|B^q u\|^{p/q} \|u\|^{1-p/q}$ and the source conditions $x^\dagger = (A^*A)^p v$, we deduce that
\begin{eqnarray}\label{PosterioriProofIneq1}
\begin{array}{ll}
\| r_{\alpha}(A^*A) x^\dagger \|
= \| (A^*A)^{p} r_{\alpha}(A^*A) v \| \\ \quad
\leq \| (A^*A)^{(p+1/2)} r_{\alpha}(A^*A) v \|^{2p/(2p+1)} \cdot \| r_{\alpha}(A^*A) v\|^{1/(2p+1)} \\ \quad
= \| A r_{\alpha}(A^*A) x^\dagger\|^{2p/(2p+1)} \cdot \| r_{\alpha}(A^*A) v\|^{1/(2p+1)}.
\end{array}
\end{eqnarray}
Since $T_*$ is chosen according to the equation $\chi(T)=0$, we derive that
\begin{eqnarray}\label{PosterioriProofIneq2}
\begin{array}{ll}
\tau  \delta = \|A x^\delta(T_*)-y^\delta\| \\ \qquad
= \left\| A r_{T^{-\theta}_*}(A^*A) x^\dagger + r_{T^{-\theta}_*}(AA^*) (y^\delta-y) \right\| \\ \qquad
\geq \| A r_{T^{-\theta}_*}(A^*A) x^\dagger \| - \|r_{T^{-\theta}_*}(AA^*) (y^\delta-y) \|.
\end{array}
\end{eqnarray}
Now we combine the estimates (\ref{PosterioriProofIneq1}) and (\ref{PosterioriProofIneq2}) to obtain, with the source conditions, that
\begin{eqnarray}\label{PosterioriProofIneq3}
\begin{array}{ll}
\| r_{\alpha}(A^*A^*)x^\dagger\|
\leq \| A r_{\alpha}(A^*A) x^\dagger\|^{2p/(2p+1)} \cdot \| r_{\alpha}(A^*A) v\|^{1/(2p+1)} \\ \quad
\leq \left( \tau  \delta + \|r_{T^{-\theta}_*}(AA^*) (y^\delta-y) \| \right)^{2p/(2p+1)} \rho^{1/(2p+1)} \\ \quad
\leq c'_2 \rho^{1/(2p+1)} \delta^{2p/(2p+1)}
\end{array}
\end{eqnarray}
where $c'_2:= \left( \tau + 3 \right)^{2p/(2p+1)}$.

On the other hand, in a similar fashion to (\ref{PosterioriProofIneq2}), it is easy to show that
\begin{equation}\label{PosterioriProofIneq2Log}
\begin{array}{ll}
\tau \delta &= \|A x^\delta(T_*)-y^\delta\| = \| r_{T^{-\theta}_*}(AA^*) y + r_{T^{-\theta}_*}(AA^*) (y^\delta-y) \| \\ &
\leq \| A r_{T^{-\theta}_*}(A^*A) x^\dagger\| + \|r_{T^{-\theta}_*}(AA^*) (y^\delta-y) \| \\ &
\leq \| A r_{T^{-\theta}_*}(A^*A) x^\dagger \| + 3\delta.
\end{array}
\end{equation}
If we combine the above inequality with the source conditions and the qualification inequality, we obtain
\begin{eqnarray}\label{PfQualificationIneq}
\begin{array}{ll}
(\tau -3)\delta \leq \| A r_{T^{-\theta}_*}(A^*A) x^\dagger\| \\ \qquad
\leq  \| (A^*A)^{p+1/2} r_{T^{-\theta}_*}(A^*A) v \| \leq 2\rho \gamma (T_*)^{-\theta(2p+1)/2},
\end{array}
\end{eqnarray}
which yields the estimate for $T_*$ in (\ref{ErrorEstimatePrioriT}).
Finally, using (\ref{errorBounds2}), the estimate for $T_*$ and (\ref{PosterioriProofIneq3}), we conclude that
\begin{eqnarray*}
\begin{array}{ll}
\| x^\delta(T_*) - x^\dagger \| \leq \| r_{\alpha}(A^*A^*) x^\dagger \| + \gamma_* \sqrt{T^\theta_*} \delta \\ \qquad
\leq c_1 \rho^{1/(2p+1)} \delta^{2p/(2p+1)} + \gamma_* \left( \frac{2\gamma}{\tau -3} \right)^{1/(2p+1)} \rho^{1/(2p+1)} \delta^{2p/(2p+1)} \\ \qquad
= C_1 \rho^{1/(2p+1)} \delta^{2p/(2p+1)}.
\end{array}
\end{eqnarray*}
This completes the proof.
\end{proof}

\section{FAR under logarithmic source conditions}
\label{logarithmic}

\setcounter{section}{4}
\setcounter{equation}{0}\setcounter{theorem}{0}

Now, we turn to the study of FAR method under conditions \eqref{logarithmicCondition}.

\begin{proposition}\label{ThmQualification2}
For all $\theta\in(0,2)$ and all $\mu>0$ the index function $\varphi_\mu(\lambda)$, defined in \eqref{logarithmicQualification}, is a qualification of the FAR method.
\end{proposition}
\begin{proof}
For all $\theta\in(0,2)$ the function $\psi(\lambda)=\lambda$ is a qualification of the FAR method, as shown in Proposition~\ref{ThmQualificationPower}.
Moreover, for arbitrary $\mu>0$ the index function $\varphi_\mu(\lambda)$ is concave for all $0<\lambda\le e^{-\mu-1}$, hence due to Proposition~\ref{pro:covered} covered by $\psi$ and consequently
also a qualification of the FAR method. This proves the proposition.
\end{proof}

Based on the above proposition, the following theorem holds as a consequence of the discussions in Example~\ref{ex:ex2} (cf.~formula \eqref{eq:estlogarithmic}).

\begin{theorem}\label{ThmPriori2}
Let $x^\delta(t)$ denote the solution of \eqref{FractionalFlow} with noisy data $y^\delta$ and noise level $\delta>0$. Then, under the assumption $x^\dagger \in \mathcal{R}(\varphi_\mu(A^*A))$ (cf.~(\ref{logarithmicCondition})) for arbitrary $\mu>0$, we have
for the FAR method and all $0 < \theta<2$ the convergence rate
\begin{equation}\label{ErrorEstimatePriori3}
\| x^\delta(T_*) - x^\dagger \| =  \mathcal{O} \left( \log^{-\mu}\left(\delta^{-1} \right) \right)  \quad \mbox{as} \quad \delta \to 0
\end{equation}
whenever the terminating time $T_*$ of FAR is chosen according to $$T_* (\delta) \sim \delta^{-\kappa/\theta},$$ for some $0<\kappa<2$,
\end{theorem}

Now, consider the a posteriori choice of regularization parameter. We start with the following lemma.

\begin{lemma}\label{IneqResidualLog}
The index function $\sqrt{\lambda}\varphi_{\mu}(\lambda)$ is a qualification of FAR, where $\varphi_{\mu}$ is defined in \eqref{logarithmicQualification}. Consequently, a constant $c_3$ exists such that
\begin{equation*}
\sup_{\lambda\in(0,e^{-2\mu-1})} \sqrt{\lambda} |r_{\alpha}(\lambda)| \varphi_\mu(\lambda) \leq c_3 \sqrt{\alpha}  \log^{-\mu}\left(\alpha^{-1} \right).
\end{equation*}
\end{lemma}

 \proof 

The assertion of the lemma follows from Proposition~\ref{ThmQualification2} in combination with the fact that $\varphi(\lambda)=\lambda$ is a qualification of FAR for all $\theta\in(0,2)$, and with the following inequality
\begin{equation*}
\sup\limits_{\lambda\in(0,e^{-2\mu-1})} \left( \sqrt{\lambda} r_{\alpha}(\lambda) \varphi_\mu(\lambda) \right)^2
\leq \sup\limits_{\lambda\in(0,e^{-2\mu-1})} \lambda r_{\alpha}(\lambda) \sup\limits_{\lambda\in(0,e^{-2\mu-1})} r_{\alpha}(\lambda) \varphi_{2\mu}(\lambda).
\end{equation*}
 \proofend 


\begin{theorem}\label{ThmPosterioriLog}
Under the assumptions of Theorem~\ref{ThmPriori2}, but for the terminating time $T_*=T_*(\delta,y^\delta)$ of FAR chosen as a solution of $\chi(T)=0$ (cf.~\ref{fractionalDP})), we have
\begin{equation}\label{ErrorEstimatePosterTLog}
T_* = \mathcal{O} \left( W_{2\mu/\theta} \left( \delta^{-2/\theta} \log^{-2/\theta}(\delta^{-1}) \right) \right)  \quad \mbox{as} \quad \delta \to 0,
\end{equation}
where $W_{a}(z)$ is the unique solution of $\zeta \log^{a}(\zeta) = z$ with respect to $\zeta$, and moreover the convergence rate
\begin{equation}\label{ErrorEstimatePosterLog}
\| x^\delta(T_*) - x^\dagger \| = \mathcal{O} \left(  \log^{-\mu}(\delta^{-1}) \right)  \quad \mbox{as} \quad \delta \to 0.
\end{equation}
\end{theorem}

 \proof 
Without loss of generality we assume that $\delta<e^{-2\mu-1}$. By combining inequality \eqref{PosterioriProofIneq2Log} and Lemma \ref{IneqResidualLog} with $\alpha$ replaced by $T^{-\theta}_*$, we obtain
\begin{equation*}
\begin{array}{ll}
(\tau -3)\delta & \leq \| A r_{T^{-\theta}_*}(A^*A) x^\dagger\| \leq c_3 T^{-\theta/2}_*  \log^{-\mu}\left( T^{\theta}_* \right) \\ & \leq c_3 \theta^{-\mu/2} T^{-\theta/2}_*  \log^{-\mu}\left( T_* \right).
\end{array}
\end{equation*}
Therefore, for $\delta\in(0,e^{-1})$, and if we set $c_4:=c_3 \theta^{-\mu/2} (\tau-3)^{-1}$, we have
\begin{eqnarray}\label{PfPosterTIneqOriginal}
T_* \log^{2\mu/\theta}(T_*) \leq c^{2/\theta}_4 \delta^{-2/\theta} \log^{-2/\theta}(\delta^{-1}),
\end{eqnarray}
which yields the estimate \eqref{ErrorEstimatePosterTLog}, as well as the inequality
\begin{equation}\label{PfPosterTIneq}
\sqrt{T^\theta_*}\leq c_4  \delta^{-1} \log^{-1-\mu}(\delta^{-1}).
\end{equation}

The estimate \eqref{ErrorEstimatePosterLog} can be driven according to inequalities \eqref{errorBounds2}, \eqref{PfPosterTIneq}, and to the Proposition \ref{ThmQualification2}.
 \proofend 

\begin{remark}
(a) To the best of our knowledge, under the logarithmic source condition, for obtaining the optimal convergence rate \eqref{ErrorEstimatePosterLog}, existing regularization methods with a posteriori regularization parameter selection methods require the time cost $T_{old} = \mathcal{O} \left( \delta^{-2} \log^{-2-2\mu}(\delta^{-1}) \right)$. However, in the case of the FAR method, according to \eqref{PfPosterTIneqOriginal} we derive for $\theta>1$ that
\begin{equation*}\label{AccelerationIneq}
T_* \ll T_* \log^{2\mu/\theta}(T_*) \leq c^{2/\theta}_4   \delta^{-2/\theta} \log^{-2/\theta}(\delta^{-1}) \ll T_{old}
\end{equation*}
for sufficient small $\delta>0$.
This means that FAR with $1<\theta<2$ yields a super accelerated optimal regularization method. (b) Similar to Theorem \ref{FARTikhonov}, the converse result for FAR under logarithmic source conditions can be established by the existing technique, proposed in e.g. \cite{Scherzer2016}.
\end{remark}

\section{Numerical investigation of the FAR method}
\label{NumericalScheme}

\setcounter{section}{5}
\setcounter{equation}{0}\setcounter{theorem}{0}

Roughly speaking, the fractional differential equation \eqref{FractionalFlow} with appropriate numerical discretization schemes for the artificial time variable and stopping rule of iteration steps yields a concrete iterative regularization method. Rather than performing a rigorous numerical analysis, we provide in this section a numerical example to demonstrate if the regularization property and the acceleration effect of the FAR carry over to its numerical realization. To this end, we consider the one-step Adams-Moulton method \cite{Diethelm1999,Diethelm2004}, namely,
\begin{equation}\label{AdamsScheme}
\left\{\begin{array}{ll}
x^{\delta,P}_{k+1} = \frac{1}{\Gamma(\theta)} \sum^k_{j=0} b_{j,k+1} A^* (y^\delta - A x^\delta_j), \\
x^{\delta}_{k+1} = \frac{1}{\Gamma(\theta)} \left( a_{k+1,k+1} A^* (y^\delta - A x^{\delta,P}_{k+1})  + \sum^k_{j=0} a_{j,k+1} A^* (y^\delta - A x^\delta_j) \right),
\end{array}\right.
\end{equation}
where $b_{j,k+1} = \frac{\Delta t^\theta}{\theta} \left[ (k-j+1)^{\theta} - (k-j)^{\theta} \right]$, and $a_{j,k+1} = \frac{\Delta t^\theta}{\theta(\theta+1)}\cdot d_{j,k+1}$,
\begin{equation*}
d_{j,k+1} = \left\{\begin{array}{ll}
\left[ k^{\theta+1} - (k-\theta)(k+1)^\theta \right] , \quad j=0, \\
\left[ (k-j+2)^{\theta+1} + (k-j)^{\theta+1} - 2 (k-j+1)^{\theta+1} \right] ,~ 1 \leq j \leq k \\
1, \quad j=k+1.
\end{array}\right.
\end{equation*}

\begin{remark}
As the Landweber iterates, the iterates $x^{\delta}_{k+1}$ of (\ref{AdamsScheme}) obviously belong to the Krylov subspace $\textrm{Span}\left\{ A^* y^\delta, ..., (A^*A)^{k-1} A^* y^\delta \right\}$. Therefore, the solution $x^{\delta}_{k+1}$ of (\ref{AdamsScheme}) can be written as $x^{\delta}_{k}=g_{k}(A^*A)A^* y^\delta$, where $g_{k}$ is a polynomial of degree $k$.
For proving the regularization property of (\ref{AdamsScheme}), one has to check that $g_{k}(\lambda)$ fulfills all three conditions in Definition \ref{gAlpha}. Moreover, in order to prove the acceleration effect of \ref{gAlpha} under source conditions (\ref{HolderSourceIntro}), one must show that for all $\lambda\in(0,\|A\|^2]$: $|r_{k}(\lambda)| (k\lambda)^p \to0$ as $k\to\infty$, where $r_{k}(\lambda)= 1- \lambda g_{k}(\lambda)$ represents the residual polynomial. However, all of these issues are not addressed here since it is out of the scope of our current paper. Rather we refer to a similar result in \cite{GongHofmannZhang2019} for the second order asymptotical regularization.
\end{remark}

Now, by employing the newly developed iterative regularization method (\ref{AdamsScheme}), we present some numerical results for the following integral equation
\begin{equation}\label{IntegralEq}
Ax(s):= \int^1_0 K(s,t) x(t) dt = y(s), \quad K(s,t)=\left\{\begin{array}{ll}
s(1-t), \quad s\leq t, \\
t(1-s), \quad s> t.
\end{array}\right.
\end{equation}
In this context, we choose $\mathcal{X}=\mathcal{Y}:=L^2[0,1]$ such that the operator $A$ is compact, selfadjoint and injective. Then the operator equation $Ax = y$ can be rewritten as   $x=-y''$ provided that
 $y\in H^2[0,1]\cap H^1_0[0,1]$. Moreover, the operator $A$ has the singular system (eigensystem) $\{\sigma_j;u_j;u_j\}_{j=1}^\infty$ with $\sigma_j=(j\pi)^{-1}$ and $u_j(t)=\sqrt{2}\sin(j\pi t)$. Furthermore, using the interpolation theory (see, e.g.,~\cite{Lions-1972}) it is not difficult to show that for $4p-1/2 \not\in \mathbb{N}$
\begin{eqnarray*}
\mathcal{R}((A^*A)^{p}) = \left\{ x\in H^{4p}[0,1]:~x^{2l}(0)=x^{2l}(1)=0,~l=0,1,...,\lfloor 2p-1/4 \rfloor \right\}.
\end{eqnarray*}

In our simulations, problem \eqref{IntegralEq} is numerically solved by the linear finite element method. Let $\mathcal{Y}_n$ be the finite element space of piecewise linear functions on a uniform grid with step size $1/(n-1)$. Denote by $P_n$ the orthogonal projection operator acting from $\mathcal{Y}$ into $\mathcal{Y}_n$. Define $A_n:= P_n A$ and $\mathcal{X}_n:= A^*_n \mathcal{Y}_n$. Let $\{\phi_j\}^n_{j=1}$ be a basis of $\mathcal{Y}_n$, then, instead of \eqref{IntegralEq}, we solve a system of linear equations $A_n x_n = y_n$ in practice, where $[A_n]_{ij}= \int^1_0 \left(  \int^1_0 k(s,t) \phi_i(s) ds \right) \phi_j(t) dt$ and $[y_n]_{j} = \int^1_0 y(t) \phi_j(t) dt$.


We consider the following two different right-hand sides for \eqref{IntegralEq}.

\begin{example}\label{ex1}
In this example let $y(s):=s(1-s)$. Then the solution is $x^\dagger \equiv 2$ and we have $x^\dagger\in \mathcal{R}((A^*A)^{p})$ for all $p<1/8$.
\end{example}

\begin{example}\label{ex2}
In this example let $y(s):=s^4(1-s)^3$. Then the solution attains the form  $x^\dagger=-6t^2(1-t)(2-8t+7t^2)$, and we have $x^\dagger\in \mathcal{R}((A^*A)^{p})$ for all $p<5/8$.
\end{example}

Both examples are tested by using the grid size $n=100$. Other parameters are: $\Delta t=19.4850, x_0=\dot{x}_0=0, \tau =3.1$. Uniformly distributed noises with the magnitude $\delta'$ are added to the discretized exact right-hand side:
\begin{equation*}\label{Data}
[y^\delta_n]_j := \left[ 1 + \delta' \cdot(2 \textrm{Rand}(x) -1) \right] \cdot [y_n]_j, \quad j=1, ..., n,
\end{equation*}
where $\textrm{Rand}(x)$ returns a pseudo-random value drawn from a uniform distribution on [0,1]. The noise level of measurement data is calculated by $\delta=\|y^\delta_n - y_n\|_2$, where $\|\cdot\|_2$ denotes the standard vector norm in $\mathbb{R}^n$. The iteration of (\ref{AdamsScheme}) is terminated according to the discrepancy principle ($\tau=3$ for all examples), i.e.
\begin{equation*}\label{discrepancy2}
\|y^\delta-Ax^{k^*}_n\|_{L^2[0,1]} \leq \tau \delta < \|y^\delta-Ax^{k}_n\|_{L^2[0,1]}, \quad 0\leq k < k^*.
\end{equation*}
Finally, to assess the accuracy of the approximate solutions, we define the $L^2$-norm relative error for an approximate solution $x^{k^*}_n$: $\textrm{L2Err}:= \|x^{k^*}_n - x^\dagger\|_{L^2[0,1]}/\|x^\dagger\|_{L^2[0,1]}$, where $x^\dagger$ is the exact solution to the corresponding model problem.

\begin{table}[!t]
\caption{Comparisons with the Landweber method, Nesterov's method, Chebyshev method, and CGNE method. }
\begin{center}
\begin{tabular}{|c|c|c|c|c|c|c|c|} \hline
\multirow{2}{*}{$\delta$} &
\multicolumn{2}{c|}{$\theta=0.5$} &
\multicolumn{2}{c|}{$\theta=0.8$} &
\multicolumn{2}{c|}{Landweber} \\
\cline{2-7}
& $k^*(\delta)$ & \textrm{L2Err} & $k^*(\delta)$ & \textrm{L2Err} & $k^*(\delta)$ & \textrm{L2Err} \\ \hline
\multicolumn{7}{|c|}{\textbf{Example 1: $y(s)=s(1-s)$}}  \\  \hline
0.01 & 4712 &  0.2694 & 1961 &  0.2684 & 712 &  0.2639 \\
0.001 & 9336 &  0.1494 & 7057 &  0.1484 & 5997 &  0.1439 \\
0.0001 & $k_{max}$ & 0.2597 &  $k_{max}$ & 0.2073  & $k_{max}$ &  0.1807 \\ \hline
\multicolumn{7}{|c|}{\textbf{Example 2: $y(s)=s^4(1-s)^3$}}  \\  \hline
0.001 & 3916 & 0.3038 & 2585 & 0.2918  & 1347 & 0.2833 \\
0.0001 & 8534 & 0.2003 & 4947 & 0.1995  & 2034 & 0.1878 \\
0.00001 & $k_{max}$ & 0.3195 & $k_{max}$ & 0.2184  & $k_{max}$ & 0.1509 \\ \hline \hline
\multirow{2}{*}{$\delta$} &
\multicolumn{2}{c|}{$\theta=1.2$} &
\multicolumn{2}{c|}{$\theta=1.5$} &
\multicolumn{2}{c|}{$\theta=1.8$} \\
\cline{2-7}
& $k^*(\delta)$ & \textrm{L2Err} & $k^*(\delta)$ & \textrm{L2Err} & $k^*(\delta)$ & \textrm{L2Err} \\ \hline
\multicolumn{7}{|c|}{\textbf{Example 1: $y(s)=s(1-s)$}}  \\  \hline
0.01 & 136 &  0.2476 & 29 &  0.2187 & 11 &  0.2204 \\
0.001 & 733 & 0.1407 & 97 &  0.1321  & 45 &  0.1325 \\
0.0001 & 8716 & 0.0997 &  1286 & 0.0673  & 657 &  0.0764 \\ \hline
\multicolumn{7}{|c|}{\textbf{Example 2: $y(s)=s^4(1-s)^3$}}  \\  \hline
0.001 & 224 & 0.2520 & 90 & 0.2019  & 58 & 0.2382 \\
0.0001 & 883 & 0.0901 & 217 & 0.0767  & 155 & 0.0704 \\
0.00001 & 4582 & 0.0287 & 328 & 0.0067  & 243 & 0.0083 \\ \hline \hline
\multirow{2}{*}{$\delta$} &
\multicolumn{2}{c|}{Nesterov} &
\multicolumn{2}{c|}{Chebyshev} &
\multicolumn{2}{c|}{CGNE} \\
\cline{2-7}
& $k^*(\delta)$ & \textrm{L2Err} & $k^*(\delta)$ & \textrm{L2Err} & $k^*(\delta)$ & \textrm{L2Err} \\ \hline
\multicolumn{7}{|c|}{\textbf{Example 1: $y(s)=s(1-s)$}}  \\  \hline
0.01 & 109 & 0.2590 & 127 &  0.2553  & 6 &  0.2213 \\
0.001 & 704 &  0.1600 & 556 &  0.1496  & 19 &  0.1783  \\
0.0001 & 4160 & 0.1025 &  4361 &  0.0897  & 38 & 0.0894 \\ \hline
\multicolumn{7}{|c|}{\textbf{Example 2: $y(s)=s^4(1-s)^3$}}  \\  \hline
0.001 & 102 & 0.2501 & 75 & 0.2434  & 6 & 0.2919  \\
0.0001 & 208 & 0.1176 & 207 & 0.0990  & 13 & 0.1814 \\
0.00001 & 1805 &  0.0280 & 2226 & 0.0196  & 15 & 0.0847 \\ \hline
\end{tabular}\label{tab:Comparison}
\end{center}
\end{table}

The results of the simulation are presented in Table \ref{tab:Comparison}, where we can conclude that, in general, the FAR method with $\theta\in(1,2)$ needs fewer iterations and offers more accurate regularized solutions. Concerning the number of iterations, the conjugate gradient method for the normal equation (CGNE, cf., e.g.~\cite{Hanke1995}) performed significantly more effectively than all of other methods. However, the accuracy of the CGNE method is considerably worse than other accelerated regularization methods, since the step size of CGNE is too large to capture the optimal point and the semi-convergence effect disturbs the iteration rather early.  Note that we set a maximal iteration number $k_{max}=200,000$ in all of our simulations.

\section{Conclusion and outlook}
\label{Conclusion}

In this paper, we have investigated the Fractional Asymptotical Regularization method (FAR) for solving linear ill-posed operator equations $Ax=y$ with compact forward operators $A$ mapping between infinite dimensional Hilbert spaces. Instead of exact right-hand side $y$, we are given noisy data $y^\delta$ obeying the deterministic noise model $\|y^\delta - y\|\leq \delta$. We have proven that under both H\"{o}lder-type and logarithmic source conditions, FAR with fractional order $\theta\in(0,2)$ exhibits an optimal regularization method. Moreover, if $\theta\in(1,2)$, FAR yields an accelerated method, namely, the optimal convergence rates can be obtained with approximately the $\theta$ root of iterations than needed for ordinary Landweber iteration. Finally, with the help of the one-step Adams-Moulton method, a new iterative regularization algorithm has been introduced.

As has been shown by this manuscript, fractional calculus can play some role for regularization schemes aimed at the stable approximate solution to general linear ill-posed problems. However, to the best of our knowledge, the literature in this direction is quite limited. The initial results in this work might be a bridge between fractional calculus and regularization theory. Of course, there are many remaining interesting problems in this topic. For instance, since the Hilbert scale can be understood as preconditioning for iterative regularization methods \cite{EggerNeubauer2005}, it seems to be necessary to derive further assertions on acceleration of the FAR method in Hilbert scales. Certainly, the optimal choice of fractional order $\theta$ and extensions to nonlinear problems are of interest. Yet, in the linear setting, new questions also arise, for example the study of the damped system
\begin{equation}\label{ConEq}
\left( ^C D^\theta_{0+} x^\delta \right) (t) + \eta \dot{x}^\delta(t)  + A^*A x^\delta(t) = A^* y^\delta,
\end{equation}
where $\theta>1$, and $\eta>0$ is the damping parameter, which may or may not depend on the artificial time $t$. A remaining natural question is whether the damping term in \eqref{ConEq} can yield a further acceleration of FAR.

\section*{Acknowledgements}

We express our gratitude to the anonymous referees whose valuable comments and suggestions allowed us to eliminate weak points of the manuscript and thus to improve the paper.

The work of Y. Zhang is supported by the Alexander von Humboldt Foundation through a postdoctoral researcher fellowship, and the work of B. Hofmann is supported by the German Research Foundation (DFG-grant HO 1454/12-1).




 \bigskip \smallskip

\parindent0em

 \it

$^{1,2}$ Faculty of Mathematics,
Chemnitz University of Technology \\
Reichenhainer Str. 39/41,
09107 Chemnitz, Germany  \\[2pt]

$^1$ School of Science and Technology,
\"{O}rebro University\\
Fakultetsgatan 1,
\"{O}rebro -- 70182, Sweden  \\[2pt]

$^{1}$  e-mail: ye.zhang@mathematik.tu-chemnitz.de \\
$^{2}$  e-mail: hofmannb@mathematik.tu-chemnitz.de \\[4pt]


\vspace{0.5cm}

This paper is now published (in revised form) in \\
Fract. Calc. Appl. Anal. Vol. 22, No 3 (2019), pp. 699-721, \\ DOI: 10.1515/fca-2019-0039, \\
and is available online at http://www.degruyter.com/view/j/fca , so always cite it with the journal's coordinates.

\end{document}